\documentclass[review]{elsarticle}
\usepackage{fullpage,wrapfig}
\journal{ }
\usepackage{amsmath,amssymb,amsthm,multirow,float,graphicx,bm,multicol}
\usepackage{geometry,tabularx}
\allowdisplaybreaks
\newtheorem{theorem}{Theorem}
\newtheorem{ThmObs}{Observation}

\newlength\Colsep
\newcolumntype{Y}{>{\centering\arraybackslash}X}

\begin{document}
	\begin{frontmatter}
		\title{Column Generation and Lazy constraints for solving the Liner Ship Fleet Repositioning Problem with cargo flows}
		\author{Robin H. Pearce, Alexis Tyler \& Michael Forbes}
		\address{School of Mathematics and Physics, University of Queensland, Australia}
		
		\begin{abstract}
			We consider an important problem in the shipping industry known as the liner shipping fleet repositioning problem (LSFRP). We examine a public data set for this problem including many instances which have not previously been solved to optimality. We present several improvements on a previous mathematical formulation, however the largest instances still result in models too difficult to solve in reasonable time. The implementation of column generation reduces the model size significantly, allowing all instances to be solved, with some taking two to three hours. A novel application of lazy constraints further reduces the size of the model, and results in all instances being solved to optimality in under four minutes. 
		\end{abstract}
		
		\begin{keyword}
			Liner Shipping\sep Network flows\sep Scheduling\sep Column Generation\sep Lazy constraints
		\end{keyword}
	\end{frontmatter}

\section{Introduction}
The shipping industry is responsible for transporting many billion tons of cargo each year, almost 9.6 billion tons in 2013, with a large percentage of this (35.4\%) made up by container shipping \cite{UNCTAD2014}. Optimising the practices of the shipping industry to maximise profits is the subject of considerable academic and commercial interest. Much of this research centres around the principle of optimising scheduling to maximise profits. The main areas of research are in the fleet deployment problem and the network design problem, which are related to this principle.

Liner shipping operates on a fixed schedule and uses a standard container size, the ISO container \cite{Tierney2014}. Within the shipping industry, ships are often referred to as vessels, and the set of ships which a liner shipping company operates is called a fleet. An operator is the person responsible for the operation of the shipping company and its fleet. These companies operate within a shipping network, which is a set of ports connected by arcs over which the fleet may travel. While liner shipping does operate on a fixed, periodic schedule, ships within the fleet often need to be moved between services to meet the demand of clients. This often happens to create new or modify existing services. However, repositioning vessels is an expensive exercise due to fuel costs and potential losses in revenue, so optimising these journeys is of particular interest within the industry. 

In this paper we present several improvements for solving the LSFRP. We begin with an improved MIP model for the LSFRP which ignores some less important aspects of the problem. Next, we apply column generation and present a proof of the integer properties of the column generation master problem. Finally, we present a novel use of lazy constraints to massively speed up the solution of the column generation formulation.

This paper is organised as follows: the remainder of this section will be a review of the literature around shipping network optimisation, followed by a review more specific to the LSFRP. In Section \ref{SecProbDescription} we will describe the problem in more detail, and in Section \ref{SecImprovedTech} we will look at the improvements we have made to the original formulation. Section \ref{SecResults} contains computational results comparing our different formulations, and Section \ref{SecDiscussion} is a discussion about various aspects of this problem and the associated public data set.

While the optimisation of network flow models is well studied, very little work focusses on the application of these models to shipping networks. This is emphasised in a survey conducted by Ronen \cite{Ronen1993} in which he states that, while the development of these models has the potential to have a large economic impact, the lack of literature in the field has meant that the application of these models in industry has been limited. Ronen details the models and techniques that were used in ship scheduling in the previous decade. He notes that the majority of the work was focused on four main areas: fleet size, mix and deployment; inventory routing; cruising speed; and ship scheduling. It is interesting to note that only a handful of the papers reviewed in this survey pertain to liner shipping. 

One such study is by Brown, Graves and Ronen \cite{Brown1987}, and focusses on scheduling ocean transportation of crude oil. The authors use an elastic set partitioning model which incorporates all fleet cost components and optimises for speed as well as scheduling. However, this model assumes that all ships are of similar sizes, and that one unit of cargo is a full shipload which has only a single discharging port, which, in reality, are not feasible assumptions. With these assumptions the authors are able to solve this problem with thousands of binary variables to an optimal integer solution in less than a minute. 

Another study is by Rana and Vickson \cite{Rana1991}, in which they use Lagrangian Relaxation to solve the container ship routing problem. Their model incorporates multiple ships and solves for the optimal schedule and amount of cargo transported by the fleet to maximise profit. The constraints in this model cover the carrying capacity of the ships, available cargo, time duration, and the requirement of having a connected feasible route for each ship. This model does not incorporate more complex scheduling constraints for the ships, for example a cargo delivery deadline, which limits the potential applications for this model.

In their 1992 study, Rathi, Church and Solanki \cite{Rathi1992} explore allocating resources to support a multi-commodity flow with time windows. This study presents three different linear programming (LP) models to minimise costs associated with this problem. All three formulations achieve this by minimising lateness, which is determined by the amount of assets of each type that should be allocated to a given route in each time period.  

The review by Christiansen, Fagerholt and Ronen \cite{Christiansen2004} provides an overview of the published research on ship routing and scheduling from 1994-2002. The review is split into several sections. The first section looks at strategic fleet planning, the design of fleets. The second section is tactical and operational fleet planning which explores ship routing and scheduling problems for different types of shipping. Fagerholt and Christiansen \cite{Fagerholt2000} solve a combined ship scheduling and allocation problem for industrial shipping using set partitioning. Next, they consider a problem related to robust ship scheduling with multiple time windows using a set partitioning method, where all feasible ship schedules are found prior to the model being solved \cite{Christiansen2002}. Sherali \cite{Sherali1999} investigates fleet management models and algorithms for an oil-tanker routing and scheduling problem using a mixed integer program (MIP) with heuristics. 

The third and final section of their review covers naval applications and other related problems. Of the papers reviewed pertaining to ship routing and scheduling, most used dynamic programming (DP), integer programs (IP) or heuristic algorithms to solve the problem. The authors note that trends in the literature display a growing need for research into these types of problems.

There are a number of studies dedicated to solving the fleet deployment problem, such as the one by Powell and Perkins \cite{Powell1997}, and the network design problem, for example those by Agarwal and Ergun \cite{Agarwal2008}, \'{A}lvarez \cite{Alvarez2009} and Brouer et. al. \cite{Brouer2013}. These problems are related to the decisions associated with the design of networks and timetables, as opposed to determining optimal paths for ships through a predetermined network. Agarwal and Ergun provide a unique formulation of ship scheduling and network design for cargo routing in liner shipping. The authors use a mixed integer linear program to solve the ship scheduling and cargo routing problem. Their model includes constraints such as frequency of operation and transshipment of cargo. 

Agarwal and Ergun also develop several solution techniques and compare their success and efficiency. The described algorithms are a greedy heuristic, column generation and Benders decomposition. They also outline an iterative search algorithm to generate schedules for liner shipping. Computational tests were performed on each of the algorithms using generated instances and it was found that, while the greedy heuristic was able to solve the problem very quickly, the solution quality was low. It was also found that the column generation and Benders' decomposition approaches achieved similar results in terms of accuracy, however the Benders' formulation was significantly faster, particularly when the instance size was large.

\subsection{Liner ship fleet repositioning problem}
The LSFRP is a type of network design problem which involves repositioning ships between service routes while maximising profit. This is achieved by visiting ports and delivering cargo while repositioning. However, despite the body of literature devoted to liner shipping and its surrounding problems, very little of this research is focused on the liner shipping fleet repositioning problem.

The first study that explores the LSFRP is by Tierney et. al. \cite{Tierney2012}. In this paper, the authors solve a simplified version of the LSFRP without cargo flows, empty equipment, or sail-on-service (SOS) opportunities (discussed further below). Tierney and Jensen then conduct a study which continues to explore this problem and incorporated cargo flows \cite{Tierney2012b}. In this paper, the authors use a mixed-integer program (MIP) in conjunction with a constructed graph to solve the LSFRP. This graph incorporates many of the LSFRP-specific constraints (such as SOS opportunities) so that they can be removed from the model formulation. This approach is able to solve several instances to optimality, however there are many larger instances where the problem can not be solved, as the solver runs out of memory or exceeds the maximum CPU time of one hour.

Another approach to solving the LSFRP is proposed by Kelareva, Tierney and Kilby \cite{Kelareva2014}. In this study, they solve the full LSFRP with SOS opportunities, however they do not incorporate cargo flows into their model. A constraint programming (CP) method is used with lazy clause generation, and is tested against the MIP in \cite{Tierney2012}. After testing the different models on a data set, the CP method is found to be faster than the MIP for all instances. However, this only occurs after choosing a search strategy for the particular problem. The authors note that, without sufficient understanding of the problem and of CP modelling techniques, it is difficult to choose a search strategy that is both fast and successfully finds optimal solutions. Furthermore, the CP method can not be extended to allow pre-computations or chaining of SOS or opportunities to carry empty cargo containers. 

The most recent study on the LSFRP is by Tierney et. al. \cite{Tierney2014}, which expands on the work by Tierney and Jensen from 2012 \cite{Tierney2012b}. They improve the model, provide a public data set, and use a heuristic approach. This model was able to incorporate many complex aspects of the LSFRP, including SOS opportunities, phase-in/phase-out requirements, and flexible arcs. Some of these (SOS opportunities and phase-in/phase-out requirements) are processed into the graph structure, along with sailing costs and cabotage restrictions. The MIP forms a ``disjoint path problem in which a fractional multi-commodity flow is allowed to flow over arcs in the vessel paths, along with a small scheduling component in the flexible nodes'' \cite{Tierney2014}. 

\section{Problem description and model formulation}\label{SecProbDescription}
The LSFRP consists of finding sequences of activities that move vessels between services in a liner shipping network, while maximising profit by trading off ship moving costs and cargo flow incomes \cite{Tierney2014}. ``Liner shipping services are composed of multiple slots, each of which represents a cycle that is assigned to a particular vessel''. The slots contain nodes or ports which must be visited by vessels at specific times in sequence. When a vessel is assigned a slot, it sails to all of its ports in order and delivers its cargo. Figure \ref{f-slots} shows an example of a service with three slots (represented by the differently styled lines in the graph) and five ports ($a,b,c,d,e$). The diagram shows that each slot takes three weeks to return to the start of the cycle, so three ships would be needed to run this service weekly. An LSFRP needs to be solved when we are transitioning a fleet from one set of services to a new one.

\begin{figure}[t]
	\centering
	\includegraphics[scale=0.6]{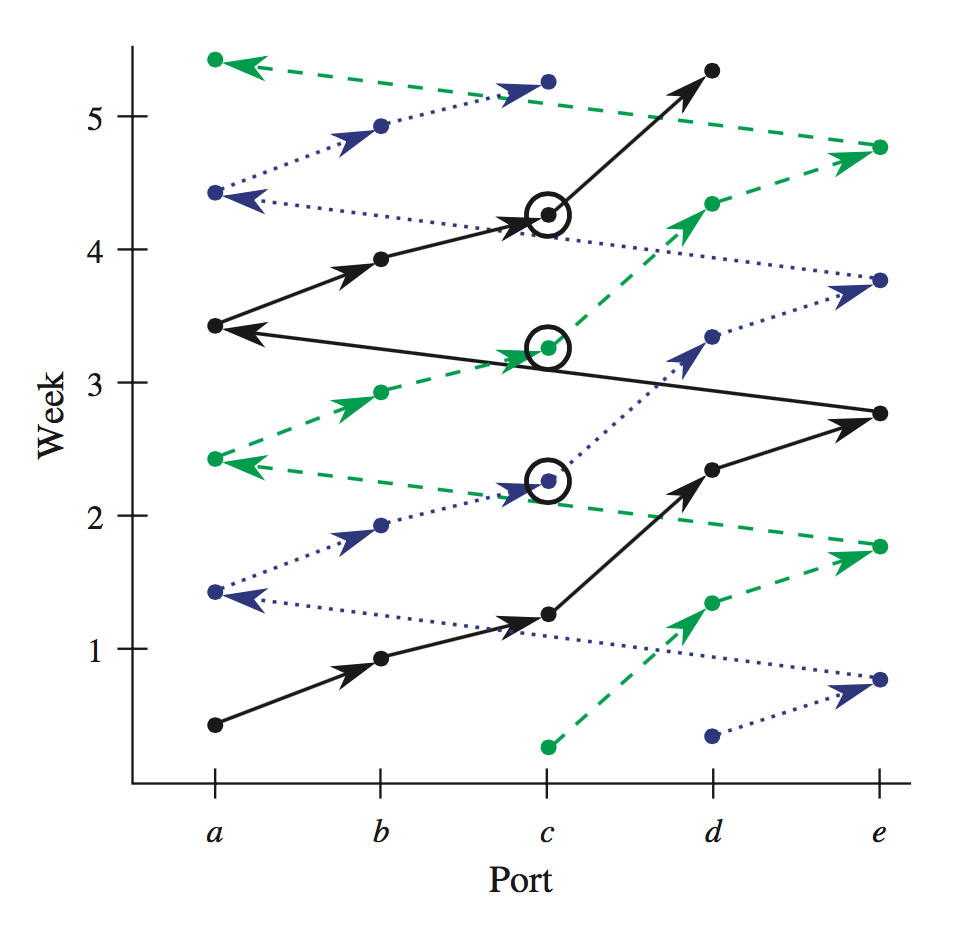}
	\caption{A time-space graph of a service with three vessels. Image sourced from Tierney et. al. (2014) \protect\cite{Tierney2014}}
	\label{f-slots}
\end{figure}

Another aspect of repositioning that needs to be taken into account is the time constraints. The time at which a ship may begin repositioning is known as the phase-out time. The ship must finish repositioning by the phase-in time of the goal service. In between these two times the ship is available for repositioning and is able to undertake a number of activities to both reach its goal service and reduce costs. An example of this can be seen in Figure \ref{f-slots}. The latest phase in time is at port c in week 2 and the phase out time is after all the hollow circles \cite{Tierney2014}. Between these points a vessel may undertake repositioning activities. 

The LSFRP is best described using Figure \ref{f-LSFRP}. This shows a ship which needs to be repositioned from its initial service (Chennai Express) to a goal service (Intra-WCSA). During repositioning the vessel can deliver cargo to ports to offset the cost of moving the ship, thus cargo flows are an important aspect of the problem. One way to do this is to take advantage of sail-on-service (SOS) opportunities, which are situations in which a repositioning ship can replace an on-service vessel for part of its service in order to reduce costs (by not having two ships sailing on the same course unnecessarily). There are two main methods of performing a SOS opportunity: transhipping, where all cargo from the on-service ship is moved onto the repositioning ship at a port, or parallel sailing, where the two ships visit the same ports sequentially and the on-service vessel only unloads cargo, while the repositioning vessel only loads cargo.

\begin{figure}[t]
	\centering
	\includegraphics[scale=0.75]{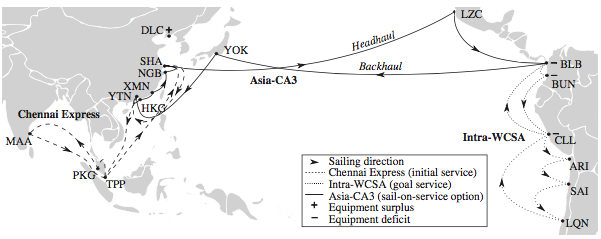}
	\caption{Liner shipping network. Image sourced from Tierney et. al. (2014) \protect\cite{Tierney2014} }
	\label{f-LSFRP}
\end{figure}

\subsubsection{Empty containers and flexible arcs}

Another way for ships to offset the cost of repositioning is to transport empty containers from ports with an empty equipment surplus to ports with a deficit. The revenue from performing this type of activity is calculated as an approximation of the savings from moving the equipment now, as opposed to at a later date, potentially through a more expensive channel. There are two types of cargo considered in this problem: dry and refrigerated (reefer). We must differentiate between the two types, since when transporting cargo the reefer containers must be plugged into a power outlet, which means that ships will only have a limited reefer capacity. This is not the case when moving empty equipment, however we still make the distinction as the deficit we are supplying may be for a specific container type. We use the term flexible visits to denote ports with empty equipment available, but no actual cargo demands. These flexible visits are travelled to via flexible arcs. 

There are also various restrictions placed on the cargo carried by repositioning ships such as trade zones. Trade zones are countries or groups of countries with trade agreements. Often cargo cannot flow between trade zones without violating these agreements. To avoid the movement of cargo violating these trade zone restrictions, the law, or a customer contract, repositioning ships are disallowed from crossing into other trade zones while carrying cargo. A similar restriction is known as a cabotage restriction, which prevents international ships from performing domestic cargo services \cite{Tierney2014}. These are all aspects which need to be considered when modelling the LSFRP. Most of these restrictions have been incorporated directly into the network of potential ship paths, so they will not be represented in the MIP formulation.

For the original model formulation, we refer the reader to the paper by Tierney et. al. (2014) \cite{Tierney2014}. We have maintained consistency in notation from previous studies. We now present a reduced formulation of the LSFRP.

\subsection{Reduced MIP}\label{SubSecRedMIP}
Starting with the MIP model described in Tierney et. al. (2014) \cite{Tierney2014}, a reduced version is formulated that does not incorporate flexible arcs or empty equipment. These aspects of the problem are omitted for simplicity to allow us to explore the core structure of the LSFRP without added complexity. By noting the percentage of the public data set that does not include these additional requirements (66\%), it can be seen that the reduced problem is still able to provide much value, as the majority of the instances do not contain the more complex aspects of the problem. There are more reasons why omitting these aspects of the problem are reasonable, some of which will be explored in Section \ref{SecDiscussion}.

\subsection*{Parameters}

\begin{align*}
&S 	&	&\textup{Set of ships.}\\
&V' 	&	&\textup{Set of visits minus the graph sink.}\\
&A' &	&\textup{Set of arcs minus those arcs connecting to the graph sink,} \\ 
&  & & \text{i.e., $(i,j) \in A, i, j \in V'$.}\\
&Q &	&\textup{Set of cargo types; $Q = \{dc,rf\}$.}\\
&M &	&\textup{Set of demand triplets of the form $(o,d,q)$, where $o \in V'$, $d \subseteq V'$, }\\
& & &\textup{and $q \in Q$ are the origin visit, possible destination visits, and the }\\
& & & \textup{cargo type respectively.}\\
&M_i^{\textup{Orig}}, (M_i^{\textup{Dest}}) \subseteq M & &\textup{Set of demands with an origin (destination) visit $i \in V$.}\\
&u_s^q \in \mathbb{R}^+ &	&\textup{Capacity of vessel $s$ for cargo type $q \in Q$.}\\	
&v_s \in V' &	&\textup{Starting visit of ship $s \in S$.}\\
&r^{(o,d,q)} \in \mathbb{R}^+ &		&\textup{Amount of revenue gained per TEU of loaded containers carried} \\
& & &\textup{for the demand triplet.}\\
&c_{sij}^{\textup{Sail}} \in \mathbb{R}^+ &	&\textup{Fixed cost of vessel $s$ utilizing are $(i,j) \in A'$.}\\
&c_{i}^{\textup{Mv}} \in \mathbb{R}^+ &	&\textup{Cost to move a single TEU on or off a ship at visit $i \in V'$.}\\
&c_{si}^{\textup{Port}} \in \mathbb{R} &	&\textup{Port fee associated with vessel $s$ at visit $i \in V'$.}\\
&a^{(o,d,q)} \in \mathbb{R}^+ &	&\textup{Amount of demand available for the demand triplet.}\\
&In(i) \subseteq V' &	&\textup{Set of visits with an arc connecting to visit $i \in V$.}\\
&Out(i) \subseteq V' &	&\textup{Set of visits receiving an arc from visit $i \in V$.}\\ 
&\tau \in V &	&\textup{Graph sink, which is not an actual visit.}
\end{align*}

\subsection*{Variables}

\begin{equation*}
\begin{aligned}
&x_{ij}^{(o,d,q)} \in \mathbb{R}^+_0 	&	&\textup{Amount of flow of demand triplet $(o,d,q) \in M$ on $(i,j) \in A'.$}\\
&y^s_{ij} \in \{0,1\} &	&\textup{Indication of whether vessel $s$ is sailing on arc $(i,j) \in A$.}\\
\end{aligned}
\end{equation*}

\subsection*{Objective and Constraints}

\begin{eqnarray}
\max &&\Bigg \{ \sum_{(o,d,q) \in M} \left( \sum_{j \in d} \sum_{i \in In(j)} (r^{(o,d,q)} - c_o^{Mv} - c_j^{Mv})x_{ij}^{(o,d,q)}\right )\label{RMIP-OBJ1}\\
&&- \sum_{s \in S}\sum_{(i,j) \in A'} c_{sij}^{\textup{Sail}} y^s_{ij} - \sum_{j \in V'} \sum_{i \in In(j)} \sum_{s \in S} c_{sj}^{\textup{Port}} y^s_{ij}\Bigg \}\label{RMIP-OBJ2}\\
\text{s.t.} &&\sum_{s \in S} \sum_{i \in In(j)} y_{ij}^s \leq 1, \quad \forall j \in V';\label{RMIP-C1}\\
&& \sum_{j \in Out(i)} y_{ij}^s = 1, \quad \forall s \in S, i=v_s;\label{RMIP-C2}\\
&&\sum_{i \in In(\tau)} \sum_{s \in S} y^s_{i \tau} = |S|;\label{RMIP-C3}\\
&&\sum_{i \in In(j)} y_{ij}^s - \sum_{i \in Out(j)} y^s_{ij} = 0, \quad \forall j\in V' \backslash \bigcup_{s \in S} v_s, s\in S;\label{RMIP-C4}\\
&&\sum_{(o,d, rf) \in M} x_{ij}^{(o,d,rf)} \leq \sum_{s \in S} u_s^{rf} y_{ij}^s, \quad \forall (i,j) \in A';\label{RMIP-C5}\\
&&\sum_{(o,d, q) \in M} x_{ij}^{(o,d,q)} \leq \sum_{s \in S} u_s^{dc} y_{ij}^s, \quad \forall (i,j) \in A';\label{RMIP-C6}\\
&&\sum_{i \in Out(o)} x_{oi}^{(o,d,q)} \leq a^{(o,d,q)}\sum_{i \in Out(o)} \sum_{s \in S} y^s_{oi}, \quad \forall (o,d,q) \in M;\label{RMIP-C7}\\
&&\sum_{i \in In(j)} x_{ij}^{(o,d,q)} - \sum_{k \in Out(j)} x_{jk}^{(o,d,q)} = 0, \quad \forall (o,d,q) \in M, j \in V' \backslash (o \cup d);\label{RMIP-C8}
\end{eqnarray}

The objective function maximises the profit of the shipping company. The first line (\ref{RMIP-OBJ1}) calculates the profit from delivering the cargo by adding the revenue minus the cost to transport the cargo on and off the ship. This is multiplied by the amount of cargo carried. The second line of the objective function (\ref{RMIP-OBJ2}) subtracts the sum of the sailing costs and the port fees for each port visited by each ship. 

Constraint (\ref{RMIP-C1}) ensures that only one ship visits each port, while (\ref{RMIP-C2}-\ref{RMIP-C4}) conserve the flow of each ship from its starting port to the sink node. If a ship uses an arc, that arc is assigned a reefer capacity in (\ref{RMIP-C5}) and a total capacity in (\ref{RMIP-C6}). Constraint (\ref{RMIP-C7}) ensures that cargo can only flow along an arc if it is on a ship. Constraint (\ref{RMIP-C8}) conserves the flow of cargo from its source node to its destination by ensuring that if it enters an intermediate node, it must also exit that node.

The reduced MIP only uses the $x_{ij}^{(o,d,q)}$ and $y_{ij}^s$ variables, as the others pertain to flexible arcs, empty equipment or entrance/exit times. As such, constraints and terms referring to the other variables are omitted from the reduced formulation. With the use of the Gurobi solver package \cite{Gurobi2015} this model was implemented and solved using a subset of the public data instances (those which did not contain flexible arcs). Pre-processing ensures that only the $x_{ij}^{(o,d,q)}$ variables which can be non-zero are added to the model. That is, there exists a path from the origin to one of the demand points passing through arc (i,j).

\section{Improved solution techniques for LSFRP}\label{SecImprovedTech}
We have made a number of improvements to the solution of the LSFRP, which we describe in this section. The reduced MIP is our starting point, which eliminates some complexity while preserving the core components of the problem. We then tighten some constraints and reformulate the model using disaggregation, which increases the number of variables we solve for, but results in a tighter bound. Next, we apply column generation, which allows us to solve all previously unsolved instances within a large time window. Finally, we apply lazy constraints, which greatly simplifies the problem, to the point where all instances are solved to optimality within four minutes.

\subsection{Tighter Bound and revised formulation}\label{SubSecTightBound}
While the reduced MIP is able to replicate the results shown in Tierney et. al. (2014) \cite{Tierney2014}, it is still unable to solve the last seven instances in the public data set. We note that one of the reasons the MIP struggles on larger problems is because the linear relaxation of the problem generates solutions in which fractional ship variables are used to transport all of a demand triplet. In order to prevent this an additional set of constraints is added:

\begin{equation*}
x_{ij}^{(o,d,q)} \leq a^{(o,d,q)} \sum_{s \in S} y_{ij}^s \quad \forall (i,j) \in A', (o,d,q) \in M;
\tag{\ref{RMIP-C7}a} \label{TB1}
\end{equation*}

These constraints prevent ships from moving a greater fraction of the demand triplet than the fraction of the ship used. This is a disaggregated version of constraint (\ref{RMIP-C7}) from the reduced formulation, as it is no longer summed over $i \in Out(o)$. This is allowed since only one ship can visit any node, and thus only one arc leaving each node will have a non-zero value of $y_{ij}^s$ in any integer solution. By the properties of disaggregation this must give a tighter bound for the linear relaxation. This improved bound yielded strong improvement on some larger instances, however it is still unable to solve five instances to optimality within the timeout limit.

Despite these tighter constraints, fractional parts of demand triplets can still be shipped, as the new constraints apply to all the ships rather than individual ships (the RHS is summed over s). To combat this, we reformulate the model for individual ships by adding a ship index to the $x$ variables. As stated earlier, an important aspect of the problem to note is that the paths need to be node distinct, meaning that only one ship can visit each node. This property means that no product can be transshipped, and allows the ship index to be added to the $x$ variables. The revised formulation with $x_{ij}^{s,(o,d,q)}$ variables is shown below.

\subsubsection*{Variables}
\begin{equation*}
\begin{aligned}
&x_{ij}^{s,(o,d,q)} \in \mathbb{R}^+_0 	&	&\textup{Amount of flow of demand triplet $(o,d,q) \in M$ on $(i,j) \in A'$ on $s \in S.$}\\
&y^s_{ij} \in \{0,1\} &	&\textup{Indication of whether vessel $s$ is sailing on arc $(i,j) \in A$.}
\end{aligned}
\end{equation*}

\subsubsection*{Objective and Constraints}
\begin{eqnarray}
\max &&\Bigg \{ \sum_{(o,d,q) \in M} \sum_{s \in S} \sum_{j \in d} \sum_{i \in In(j)} (r^{(o,d,q)} - c_o^{Mv} - c_j^{Mv})x_{ij}^{s,(o,d,q)}\label{RFObj1}\\
&&- \sum_{s \in S}  \sum_{(i,j) \in A'} c_{sij}^{\textup{Sail}} y^s_{ij} - \sum_{j \in V'} \sum_{i \in In(j)} \sum_{s \in S} c_{sj}^{\textup{Port}} y^s_{ij}\Bigg \}\label{RFObj2}\\
\text{s.t.} &&\sum_{s \in S} \sum_{i \in In(j)} y_{ij}^s \leq 1, \quad \forall j \in V';\label{RFC1}\\
&& \sum_{j \in Out(i)} y_{ij}^s = 1, \quad \forall s \in S, i=v_s;\label{RFC2}\\
&&\sum_{i \in In(\tau)} \sum_{s \in S} y^s_{i \tau} = |S|;\label{RFC3}\\
&&\sum_{i \in In(j)} y_{ij}^s - \sum_{i \in Out(j)} y^s_{ji} = 0, \quad \forall j\in V' \backslash \bigcup_{s \in S} v_s, s\in S;\label{RFC4}\\
&&\sum_{(o,d, rf) \in M} x_{ij}^{s,(o,d,rf)} \leq u_s^{rf} y_{ij}^s, \quad \forall (i,j) \in A', s \in S;\label{RFC5}\\
&&\sum_{(o,d, q) \in M} x_{ij}^{s,(o,d,q)} \leq  u_s^{dc} y_{ij}^s, \quad \forall (i,j) \in A', s \in S;\label{RFC6}\\
&&\sum_{i \in Out(o)} x_{oi}^{s,(o,d,q)} \leq a^{(o,d,q)}\sum_{i \in Out(o)} y^s_{oi}, \quad \forall (o,d,q) \in M, s \in S;\label{RFC7}\\
&& \sum_{i \in In(j)} x_{ij}^{s,(o,d,q)} - \sum_{k \in Out(j)} x_{jk}^{s,(o,d,q)} = 0,\quad \forall (o,d,q) \in M, j \in V' \backslash (o \cup d), s \in S;\label{RFC8}\\
&& x_{ij}^{s,(o,d,q)} \leq y_{ij}^s \min(a^{(o,d,q)}, u_s^q), \quad \forall (i,j) \in A', s \in S, (o,d,q) \in M;\label{RFC9}
\end{eqnarray}

The objective value is unchanged from the reduced MIP, except that now the $x$ variables are also summed over all ships $s\in S$. Constraints (\ref{RFC1}-\ref{RFC4}) are identical to the original formulation, and constraints (\ref{RFC5}-\ref{RFC8}) are disaggregated versions of constraints (\ref{RMIP-C5}-\ref{RMIP-C8}), so there is now one constraint for each ship. Finally, constraint (\ref{RFC9}) is a disaggregated version of constraint (\ref{TB1}), which ensures that for each ship, and on each arc, no more cargo can be transported than is either available or able to be transported on the ship.

While this formulation does introduce more variables into the problem, it also provides a linear relaxation with a tighter bound, which allows it to solve much faster for larger instances. The results of this new MIP formulation are reported in Section \ref{SubSecRevForm}.

\subsection{Column Generation}\label{SubSecColGen}
One of the main reasons that larger instances are unsolvable when using the presented MIP formulations is because the problems are too large to consider all variables explicitly and still be solved within a reasonable time frame. The last three instances are so large that the LP relaxation can not be solved using any of the MIP formulations within one hour. To combat this issue, composite variables and the technique of column generation are employed to reduce the number of variables considered at one time \cite{Barnhart1998}.

First the MIP is decomposed into a master problem and sub-problems (one for each ship). The master problem is an IP with every possible composite variable, which represent paths of ships through the network. This is then transformed into the reduced master problem (RMP), which contains a reduced set of the composite variables. We begin by solving the LP relaxation of the RMP, which is an LP that contains two constraints. The first says each ship is used exactly once, and the second says each node is used at most once. The variables for the master problem are paths through the graph for certain ships. The sub-problems are equivalent to the revised MIP on a ship-by-ship basis, with the dual variables associated with the master problem in the objective function. Since the path of each ship is decided by the master problem, the constraints on each node being visited only once are removed from the sub-problems. Since there is one sub-problem per ship, we do not need to sum over the ships. The RMP is described as:

\paragraph*{Parameters}

\begin{align*}
&P 	&	&\textup{Set of paths.}\\
&C_{sp} \in \mathbb{R}^+&	&\textup{Profit of vessel $s$ sailing on path $p$ (revenue from moving product} \\
&&&\textup{less the cost of the path).}\\
&\delta_{isp} \in \{0, 1\}&	&\textup{1 if vessel $s$ sailing on path $p$ goes through node $i \in V'$.}
\end{align*}

\paragraph*{Variables}
\begin{equation*}
\begin{aligned}
&Z_{sp} \in \{0, 1\}		&& \textup{1 if vessel $s$ sails on path $p \in P$, 0 otherwise.}
\end{aligned}
\end{equation*}

\paragraph*{Objective and Constraints}
\begin{eqnarray}
\max && \sum_{p \in P} \sum_{s \in S} C_{sp} Z_{sp} \\
\text{s.t.} && \sum_{p \in P} Z_{sp} = 1, \quad \forall s \in S\\
&& \sum_{p \in P} \sum_{s \in S} \delta_{isp}Z_{sp} \leq 1, \quad \forall i \in V'
\end{eqnarray}

After setting up the models, an initial solution to the master problem is generated. This is achieved using a modified version of the sub-problems. First the number of possible paths through the graph for each ship is calculated, then the ships are ordered from those with the least to those with the most paths. The sub-problems are solved in this order using the objective function from the revised formulation, with $s$ as a constant. After each subproblem is solved, the nodes used in the solution are excluded from subsequent sub-problems. This heuristic procedure resulted in a feasible starting solution for each instance, though our code does allow for high cost dummy paths direct from source to sink if this procedure fails for any ship. The paths generated from these sub-problems are then added to the master problem as columns. The master problem is then solved using these columns. 

In the main loop of the algorithm, the sub-problems are solved in the reverse order than before, and if a new column is found, it is added to the master problem, which is then re-solved. If all sub-problems are solved without any new column being added to the master problem, the algorithm terminates and the optimal solution has been found. Due to specific properties of this problem, the column generation master problem gives integer optimal solutions, so a branch and bound algorithm is not required. The reasons for this are explored in Section \ref{SubSecIntegerProof}.

\subsection{Lazy Constraints}\label{SubSecLazy}
The final improvement we present for the solution of the LSFRP involves using lazy constraints to dramatically reduce the size of the column generation sub-problems. Each node can be visited by at most one ship, the ship movement graph is acyclic, and no transshipment can occur. Therefore, if any two demands are ever carried together at any time, then they will always be carried together. This is because each demand has a single port of origin, and that port cannot be visited twice.

Consider a ship that visits port A and loads some demand from that port. The ship then proceeds to port B where it loads another demand from this port. Finally, the ship then proceeds to port C. Because the demands from ports A and B were carried together, the sum of demands from A and B cannot exceed the capacity of the ship on this arc. If demand B is unloaded at port C, it will not be possible to return to port A to load more demand of type A, and vice versa.

This means that rather than solving for the flow along each arc, we can instead solve for the total flow of each demand, and add extra constraints enforcing shared capacity where it is exceeded. We define a new set $\bar{M}_s$, which is the set of all demands $(o,d,q)$ which can be moved by ship $s$. By extension, $\bar{V}^\textup{Orig}_{sq}$ is the set of all nodes from which ship $s$ can pick up a demand of type $q$ from $\bar{M}_s$, that is:

\begin{equation*}
\bar{V}^\textup{Orig}_{sq} = \{o | (o,d,q) \in \bar{M}_s\}
\end{equation*}

Finally, we denote the set of all arcs $(i,j) \in A'$ across which a demand triple $(o,d,q)\in M$ can possibly travel as $A^{(o,d,q)}$. We also use the set $M_i^{\text{Orig}}$ from the original formulation, which is the set of demands with an origin visit $i\in V$. Starting with the revised formulation from Section \ref{SubSecTightBound}, we modify the $x$ variables so that they are no longer indexed by arc. That is:

\begin{equation*}
\begin{aligned}
&x_{s}^{(o,d,q)} \in \mathbb{R}^+_0 	&	&\textup{Amount of flow of demand triplet $(o,d,q) \in M$ on ship $s \in S.$}
\end{aligned}
\end{equation*}
We remove constraints (\ref{RFC5}-\ref{RFC9}) and replace them with the following:

\begin{eqnarray}
&&\sum_{(k,d,q) \in M_k^\textup{Orig}} x_{s}^{(k,d,q)} \leq \sum_{\substack{j\in Out(k)\\(k,j)\in A'}} u_s^{dc} y_{kj}^s, \quad \forall k \in \bigcup_{q \in Q}\bar{V}^\textup{Orig}_{sq},  s \in S;\label{LC1}\\
&&\sum_{(k,d,rf) \in M_k^\textup{Orig}} x_{s}^{(k,d,rf)} \leq \sum_{\substack{j\in Out(k)\\(k,j)\in A'}} u_s^{rf} y_{kj}^s, \quad \forall k \in \bar{V}^\textup{Orig}_{s,rf}, s \in S; \label{LC2}\\
&&x_{s}^{(o,d,q)} \leq \textup{min}\left(a^{(o,d,q)},u_s^{q}\right)\sum_{\substack{j \in Out(o) \\ (o,j)\in A^{(o,d,q)}}}y_{oj}^s \quad \forall (o,d,q) \in M, s \in S;\label{LC3} \\ 
&&x_{s}^{(o,d,q)} \leq \textup{min}\left(a^{(o,d,q)},u_s^{q}\right)\sum_{j\in d}\sum_{\substack{i \in In(j)\\ (i,j)\in A^{(o,d,q)}}} y_{ij}^s \quad \forall (o,d,q) \in M, s \in S;\label{LC4}
\end{eqnarray}

Constraints (\ref{LC1}-\ref{LC2}) ensure that the sum of demands loaded at any node cannot exceed the capacity of the ship, and specifically for reefer cargo. Constraints (\ref{LC3}-\ref{LC4}) ensure that a demand can only be carried if the ship passes through the demand's origin and one of its destinations, and caps the flow of each demand by the minimum of the demand's availability and the ship's capacity for the specific type. These constraints are the bare-basic constraints that ensure that cargo can only be carried if the ship visits the origin and a destination of the demand, and empty ships will not load more than their capacity. However, if a ship is already carrying other demands, the capacity of the ship might be breached without violating any of these constraints. This is where we add lazy constraints.

While solving a sub-problem, candidate integer solutions will be found by the MIP solver at nodes in the branch-and-bound tree. When a candidate solution is found, we check to see if the capacity of the ship is violated at any stage, specifically at any location where additional demands are loaded. We follow the path of the ship through the network, keeping track of all demands that are currently on the ship (i.e. those that have been loaded at their origin but not yet unloaded at a destination), and if, at a particular node $i$, the sum of the demands currently on the ship exceeds the capacity of the ship (for any type), we then add a constraint of the form:

\noindent\begin{minipage}{\textwidth}
	\begin{minipage}[c][5cm][c]{\dimexpr0.5\textwidth-0.5\Colsep\relax}
		\begin{equation}
		\sum_{\substack{(o,d,q)\in M \\ j \in Out(i) \\ (i,j)\in A^{(o,d,q)}}} x_s^{(o,d,q)} \leq u_s^{dc} \label{LCC1}
		\end{equation}
	\end{minipage}\hfill
	\begin{minipage}[c][5cm][c]{\dimexpr0.5\textwidth-0.5\Colsep\relax}
		\begin{equation}
		\sum_{\substack{(o,d,rf)\in M\\ j \in Out(i) \\ (i,j)\in A^{(o,d,rf)}}} x_s^{(o,d,rf)} \leq u_s^{rf} \label{LCC2}
		\end{equation}
	\end{minipage}%
\end{minipage}

that is, the sum of the flow of all demands that can pass through node $i$ must be less than the capacity of the ship, and also specifically for reefer demands. After a sub-problem has been solved, we take the lazy constraints that were generated and add them to the sub-problem as regular constraints for the next iteration. By implementing these constraints, the model is now significantly smaller and easier to solve, as can be found in Section \ref{SubSecModelSize}. The results of running this implementation can be found in Section \ref{SubSecLazyRes}.

\subsubsection{Splitting demand triples}\label{SubSubSecSplitDemands}
While it's true for a specific path through the network that if two demands are carried together at any time, then they will be carried together until one is unloaded, that does not mean that the two demands must be carried together in the first place. This can occur if one of the demands has more than one possible destination. Consider the example in Figure \ref{PairDemands}, where demand A will be picked up from Origin A. There are two choices: either proceed directly to Origin B, still carrying cargo A, and pick up cargo B, or proceed to destination A1, unload cargo A and continue to Origin B. In the first case, a constraint limiting the combined capacity of demands A and B may be imposed, where it is unnecessary if the second option is chosen. 

\begin{figure}[b]
	\centering
	\includegraphics[width=0.75\textwidth]{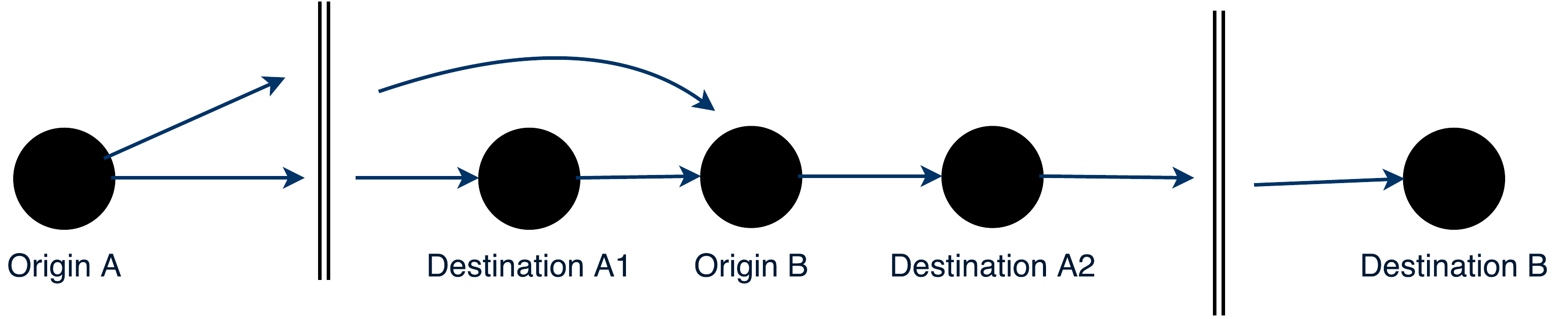}	\label{PairDemands}
	\caption{A scenario where the variables for a particular demand triple need to be separated.}
\end{figure}

This means that the variables for demand A must be separated, which involves looking for any demand triples $(o,d,q)$ that fit the following criteria:

\begin{itemize}
	\item There are two destinations $d_1, d_2 \in d$ such that $d_2$ is reachable from $d_1$.
	\item There exists another demand $m^*$ with origin $o^*$ such that $o^*$ is reachable from $d_1$, and $d_2$ is reachable from $o^*$.
	\item There exists a destination $d^*$ of demand $m^*$ such that $d^*$ is reachable from $d_2$.
	\item There exists a path between the origins $o$ and $o^*$ that does not pass through $d_1$.
\end{itemize}

If such conditions are met, then the variables for $x_{s}^{(o,d,q)}$ are split up into $x_{s}^{(o,d_i,q)}$ $\forall d_i \in d$, and an additional constraint is added:

\begin{equation}
\sum_{d_i \in d} x_{s}^{(o,d_i,q)} \leq a^{(o,d,q)}
\end{equation} 

This ensures that we will not add any lazy constraints to the problem which will unnecessarily over-constrain the problem. Each demand triple has an associated revenue and amount, both of which will be inherited by the split variables.

\subsubsection{Modification to the objective function}\label{SubSubSecLazyObj}
Because we are no longer explicitly calculating which arcs the cargo travels along, we cannot easily determine which destination it is being delivered to. Our implementation assumes that cargo is unloaded at the earliest possible time, that is, if the node $i$ is visited, all demands $(o,d,q)$ where $i\in d$ will be unloaded. In the public data set, each destination in a demand triple is the same physical port, the difference is the delivery time. This means that the unloading cost of each demand triple is the same for all destinations. This allows us to take the unload cost of any destination from the demand triple. If the destinations had different unload costs, we would have to split the demand triple as discussed in Section \ref{SubSubSecSplitDemands}.

\section{Results}\label{SecResults}
Here we consider the results from Tierney et. al. (2014) \cite{Tierney2014} and Tyler (2015) \cite{Tyler2015}, and compare them to our results. We use the same code as Tyler, however since their results were calculated, Gurobi \cite{Gurobi2015} released version 6.5 which led to significant improvements to the solve time of large MIPs. We are using a similar computer to Tyler for the reduced MIP, and report their comparison to Tierney et. al., as well as our comparison to Tyler. The machine on which the optimisation was run uses Windows 8.1 Enterprise, Python 2.7.10 and Gurobi 6.5, with an Intel i7-3770 (3.40 GHz) running 8 threads with 8 GB of RAM. All software involved is 64-bit.

Tierney et. al. tested this formulation on a public data set of 44 instances of increasing complexity. The number of ships ranged from 3 to 11, with between 30 and 379 ports and 94 to 11979 arcs. With a time limit of 1 hour and a memory limit of 10GB, the MIP implemented in CPLEX 12.4 was able to solve the first 33 instances to optimality.

\subsection{Comparison of reduced model with original}\label{SubSecOrigComp}
Table \ref{t-replicate} shows the results of the reduced MIP compared with the results reported in the paper by \citeauthor{Tierney2014}. Instances are left to run for 1 hour, and if no solution is found within that time they were said to have timed out. Instances that run out of memory are denoted by `Mem'. In the table $|M|$ represents the number of demands, $|V|$ the number of ports, $|A|$ the number of arcs and $|S|$ the number of ships. 
\begin{table}
	\centering
	\caption{Results of reduced MIP run on public data instances from Tierney et. al. (2014) \protect\cite{Tierney2014}  \label{t-replicate}}
	\scalebox{0.75}{
		\begin{tabular}{|c|*{7}{c|}}
			\hline
			Instance& $|S|$ & $|V|$ & $|A|$ &  $|M|$  & Tierney MIP & Tyler MIP & Reduced MIP \\ 
			ID & & & &  &time (s) & time (s) & time (s) \\ \hline
			repos1p & 3 &36 &150&28&0.06 & 0.07 & 0.23   \\ \hline
			repos2p & 3 &36 &150&28&0.06 & 0.07 & 0.38  \\ \hline
			repos3p & 3 &38 &151&24&0.04 & 0.06 & 0.25 \\ \hline
			repos4p & 3 &42 &185&20&0.04 & 0.07 & 0.30 \\ \hline
			repos5p & 3 &51 &270&22&0.07 & 0.11 & 0.29 \\ \hline
			repos6p & 3 & 51&270&22&0.08 & 0.11 & 0.27 \\ \hline
			repos7p & 3 & 54&196&46&0.08 & 0.09 & 0.26 \\ \hline
			repos10p & 4 & 58&499&125&74.85 & 7.25 & 4.26 \\ \hline
			repos12p & 4 & 74&603&145&106.63 & 14.94 & 11.94 \\ \hline
			repos13p & 4 & 80&632&155&99.81 & 16.53 & 10.66 \\ \hline
			repos15p & 5 & 71&355&173&0.47 & 0.29 & 0.40  \\ \hline
			repos16p & 5 & 106&420&320&1.08 & 0.42 & 0.52  \\ \hline
			repos17p & 6 & 102&1198&75&4.64 & 1.55 & 1.52 \\ \hline
			repos18p & 6 & 135&1439&87&6.79 & 1.25 & 1.25   \\ \hline
			repos20p & 6 & 142&1865&80&13.84 & 3.34 & 2.86 \\ \hline
			repos24p & 7 & 75&482&154&2.23 & 0.47& 0.58  \\ \hline
			repos25p & 7 & 77&496&156&3.19 & 0.57 & 0.73  \\ \hline
			repos27p & 7 & 79&571&188&1394.44 & 4.43& 1.79   \\ \hline
			repos28p & 7 & 90&618&189&1099.87 & 6.92 & 1.91 \\ \hline
			repos30p & 8 & 126&1450&265&307.12 & 8.15 & 4.59 \\ \hline
			repos31p & 8 & 130&1362&152&57.4 & 19.02 & 9.82  \\ \hline
			repos32p & 8 & 144&1501&170&65.51 & 10.38 & 6.96  \\ \hline
			repos34p & 9 & 304&10577&344&Time & Time & Time   \\ \hline
			repos36p & 9 & 364&11972&1048&Mem  & Time & Time \\ \hline
			repos39p & 9 & 379&11666&1109&Mem  & Time & Time   \\ \hline
			repos41p & 10 & 249&8051&375&Time & Time & Time   \\ \hline
			repos42p & 11 & 279&6596&1423&Time & Time & Time   \\ \hline
			repos43p & 11 & 320&13058&1013&Mem  & Time & Time   \\ \hline
			repos44p & 11 & 328&13705&1108&Mem  & Time & Time  \\ \hline
		\end{tabular}
	}
	\label{MIP}
\end{table}

As can be seen in Table \ref{t-replicate} the reduced MIP formulation was able to replicate the results from Tierney et. al. for all instances. The reduced MIP appears to run much faster than that used by Tierney et. al. (2014) \cite{Tierney2014}, however this is believed to be caused by an updated version of Gurobi \cite{Gurobi2015} and differences in computing power, rather than differences in the model itself. Since only the smaller models are solved within the time limit, there is not much difference between the run times reported in Tyler (2015) \cite{Tyler2015} and our run times. The times achieved by our reduced MIP will serve as a baseline for comparison with the other models we present.

\subsection{Revised Formulation with tighter bound}\label{SubSecRevForm}
As can be seen in Table \ref{t-lazycon}, the revised formulation with the tighter bound enables the MIP to solve six more instances to optimality. All solved instances are solved within 10 minutes using this formulation, which does not appear to have a significant effect on the run time of the model for smaller instances. The main benefit is that all but three of the hardest instances are solved to optimality, something which has not been achieved before. Therefore the revised formulation is a significant improvement over previous formulations for the larger instances of the LSFRP.

\subsection{Column Generation}\label{SubSecCGRes}
Table \ref{t-lazycon} shows the results of using column generation implemented in Python using Gurobi. Tyler's results using column generation report the approach is able to find optimal solutions to all instances in under 7.5 hrs \cite{Tyler2015}. The update to Gurobi 6.5 results in all instances being solved to optimality in under 2.7 hours. With the exception of repos43p and repos44p, all instances are solved within the timeout limit, with the slowest taking approximately 20 minutes, and the majority solving within minutes or seconds. For the larger instances (repos34p-repos44p), column generation is significantly faster than the other approaches (excluding repos41p which is slightly slower than the revised MIP). In the best case (repos34p) the column generation solves the problem almost three times faster than the revised formulation.

For the two instances that remain unsolved after the timeout limit, but can be solved within 8 hours, Tyler performs a comparison test with the revised MIP formulation \cite{Tyler2015}. This MIP is allowed to run on these instances with a new timeout limit of 8 hours, in order to determine whether it is able to solve these problems in a comparable time frame. However, after 8 hours both problems are still in the preprocessing stage of Gurobi's solver, and have not obtained a solution or an optimality gap. Therefore it can be concluded that for these large instances the only feasible technique for solving them is column generation. 

\subsection{Lazy constraints}\label{SubSecLazyRes}
All instances are solved to optimality using the lazy constraints formulation, and the optimal objective values all agreed with past results. The only comparison we are interested in is the difference in run-time between the different formulations. Table \ref{t-lazycon} is a comparison between all methods presented in this paper.
\begin{table}
	\centering
	\caption{Comparison of solution times for different implementations \label{t-lazycon}}
	\label{TableLabel}
	\scalebox{0.75}{
		\begin{tabular*}{\textwidth}{@{\extracolsep{\fill}}|c|c|c|c|c|}
			\hline Instance& Reduced MIP &Revised MIP &Column Generation& Lazy Constraints\\ 
			ID &time (s) &time (s)& time (s)& time (s) \\ \hline
			repos1p & 0.23 & 0.25 & 1.00 & 0.44\\ \hline 
			repos2p & 0.38 & 0.28 & 0.53 & 0.50\\ \hline 
			repos3p & 0.25 & 0.25 & 0.30 & 0.33\\ \hline 
			repos4p & 0.30 & 0.34 & 0.78 & 0.75\\ \hline 
			repos5p & 0.29 & 0.25 & 0.33 & 0.44\\ \hline 
			repos6p & 0.27 & 0.30 & 0.34 & 0.34\\ \hline 
			repos7p & 0.26 & 0.26 & 0.66 & 0.44\\ \hline  
			repos10p & 4.26 & 1.37 & 3.06 & 1.02\\ \hline 
			repos12p & 11.94 & 1.88 & 4.58 & 1.47\\ \hline 
			repos13p & 10.66 & 2.02 & 3.75 & 1.65\\ \hline 
			repos14p & 10.34 & 2.07 & 3.61 & 1.64\\ \hline 
			repos15p & 0.40 & 0.77 & 1.76 & 1.64\\ \hline 
			repos16p & 0.52 & 1.03 & 2.27 & 1.42\\ \hline 
			repos17p & 1.52 & 1.85 & 3.06 & 1.80\\ \hline 
			repos18p & 1.25 & 2.50 & 3.36 & 2.25\\ \hline 
			repos19p & 1.18 & 2.54 & 3.32 & 2.36\\ \hline 
			repos20p & 2.86 & 3.07 & 5.69 & 2.38\\ \hline 
			repos21p & 3.14 & 3.13 & 5.64 & 2.70\\ \hline 
			repos22p & 3.33 & 3.13 & 6.20 & 2.86\\ \hline  
			repos24p & 0.58 & 1.23 & 2.19 & 1.47\\ \hline 
			repos25p & 0.73 & 1.35 & 1.80 & 1.69\\ \hline 
			repos26p & 0.63 & 1.34 & 2.10 & 1.38\\ \hline 
			repos27p & 1.79 & 2.63 & 3.59 & 1.91\\ \hline 
			repos28p & 1.91 & 2.73 & 3.38 & 2.13\\ \hline 
			repos29p & 2.34 & 2.85 & 3.52 & 1.70\\ \hline 
			repos30p & 4.59 & 3.09 & 5.83 & 2.78\\ \hline 
			repos31p & 9.82 & 6.95 & 12.19 & 6.88\\ \hline 
			repos32p & 6.96 & 3.89 & 6.42 & 3.19\\ \hline 
			repos34p & Time & 251.01 & 91.51 & 19.04\\ \hline 
			repos36p & Time & 493.70 & 228.88 & 38.25\\ \hline 
			repos37p & Time & 447.98 & 263.02 & 33.63\\ \hline 
			repos39p & Time & 539.08 & 374.75 & 35.12\\ \hline 
			repos40p & Time & 541.28 & 370.84 & 29.06\\ \hline 
			repos41p & Time & 88.84 & 113.24 & 17.00\\ \hline 
			repos42p & Time & Time & 1271.69 & 57.41\\ \hline 
			repos43p & Time & Time & Time (9632.16) & 223.21\\ \hline 
			repos44p & Time & Time & Time (8924.85) & 222.63\\ \hline 
		\end{tabular*}
	}
\end{table}

\begin{table}
	\centering
	\caption{Number of ships affected by lazy constraints, and number of constraints added for each instance.}\label{t-LazyAdded}
	\begin{tabularx}{0.8\textwidth}{|c|Y|Y|Y|Y|}
		\hline  & \multicolumn{2}{|c|}{Number of Ships} & \multicolumn{2}{|c|}{Number of Constraints}  \\ 
		\hline Instance & Dry & Reef & Dry & Reef \\ 
		\hline repos10p-repos14p & 1 & 0 & 4 & 0 \\ 
		repos15p-repos16p & 0 & 3 & 0 & 7 \\ 
		repos36p & 0 & 1 & 0 & 1 \\ 
		repos39p & 0 & 1 & 0 & 2 \\
		repos42p & 3 & 10 & 8 & 21 \\
		repos43p & 2 & 10 & 4 & 22 \\
		repos44p & 2 & 10 & 5 & 27 \\ 
		\hline 
	\end{tabularx} 
	
\end{table}

Because of the added work setting up column generation, the smaller instances for the lazy constraints method all take longer to run compared to the reduced MIP. The real strength of this formulation is that it scales far better than any other formulation presented previously. Where instances repos43p and repos44p take hours to solve with column generation, using lazy constraints means the problems can be solved within four minutes. Since these larger instances are more likely to be useful in real-world situations, this formulation is by far the best for solving the LSFRP.

Another interesting result is the number of lazy constraints added during the solve process. For almost three in four instances, no lazy constraints are added at all. This means that the small formulation without flow variables for every arc is sufficient to solve the problem. The instances which most benefited from the lazy constraints formulation (repos42p-repos44p) only need at most 30 constraints for any cargo type to solve the problem to optimality. Compare this to the number of constraints that were removed from the problem, and it becomes clear why lazy constraints are so powerful when applied to this problem. A comparison of model sizes can be found in Section \ref{SubSecModelSize}.

Table \ref{t-LazyAdded} shows, for the instances where lazy constraints are added, how many ships need them and how many constraints are added in total for the different cargo types. For all instances except the three largest, no more than seven additional constraints are required before an optimal solution was found. The reefer cargo is also more likely to have lazy constraints imposed because there are significantly fewer reefer spots on each ship compared to the dry cargo.

\section{Discussion}\label{SecDiscussion}
\subsection{Proof of integer solutions to column generation}\label{SubSecIntegerProof}
One reason column generation is so effective for this problem is because it always generates integer optimal solutions. This means we don't need to use a branch-and-bound algorithm, which makes the whole solution process much faster. After making a few observations about the nature of our problem, we will prove that solutions to the column generation are always integer.

\begin{ThmObs}
	The LSFRP with a single ship type and no product flow can be modelled as a pure network flow problem.
\end{ThmObs}

\begin{figure}
	\centering
	\includegraphics[width=0.8\textwidth]{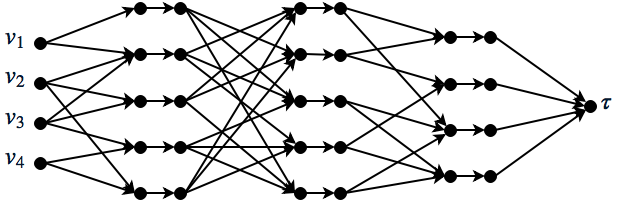}
	\caption{Example of network modification to turn column generation master problem into a pure network flow problem}\label{f-proof}
\end{figure}

``The pure network flow problem can be defined by a given set of arcs and a given set of nodes with known upper bounds and cost parameters for each arc, and fixed external flow for each node'' \cite{Jenson1999b}. Now consider the LSFRP with a single ship type. This implies that the cost of travelling along any arc is the same for all ships. If we ignore the product flow in this problem, then we can consider the problem to have a set of source nodes $v_s$ which each have a supply of 1 unit, a sink node $\tau$ which has a demand equal to the number of ships, and a set of intermediate nodes. These nodes can be split into In and Out nodes, with an arc connecting them which has a capacity of 1. The graphical representation of this can be seen in Figure \ref{f-proof}. We are now solving a pure network flow problem where all constant terms are integer, which will give us the node-distinct paths of minimum cost, along which the ships will travel.

\begin{ThmObs}
	Pure network flow problems have integer values at all of their extreme points
\end{ThmObs}

It is known that, if the constant values of a network flow problem are all integer, then the extreme points will be integer-valued \cite{Jenson1999a,Jenson1999b,Dantzig1963,Nananukul2008}.

\begin{theorem}
	The extreme points of the column generation master problem for a single ship type with or without cargo flows are all integer valued. 
\end{theorem}

\begin{proof}
	Consider first the case without cargo flows and with one ship type. Since we can reformulate the master problem as a pure network flow problem with integer-valued constraints, the extreme points will all be integer. If we now consider the case where we have product flow, the constraints of the master problem remain unchanged. While the objective function will change, which means a different extreme point may be optimal, all extreme points are integer, so the column generation solution will still be integer.
\end{proof}

For Theorem 1 to apply, we require a ``single ship type''. This means that all ships have the same capacity and costs for all arcs. We note that this is the case for all but three instances in the public data set. In each of these instances there is one ship that is different. While it is possible that these could give non-integer solutions to the LP relaxation of the master problem, in practice they did not. To handle this possibility, we would need to modify our code to branch on the ship type that visits a node. Any continuous solution that has only one ship type visiting each node can be split into a separate solution for each ship type. Theorem 1 would then hold for each ship type, and thus the optimal solution will be integer.

\subsection{Empty cargo flows in the public data set}
As described in Tierney et. al. (2014) \cite{Tierney2014}, the next step in this problem is to include the possibility of moving empty cargo containers for extra profit. This is a natural extension of the model proposed previously, except now instead of having a one-to-many delivery system as with the standard demands, we have a many-to-many system. This is because the capacity constraints on empty containers treat them as identical. The only distinction between the container types comes from the revenue earned, assuming one is more valuable than another. Our formulation of empty cargo flows is similar to that of the regular demands, except now we have one flow variable for each pair of supply-demand ports, where the demand port is reachable from the supply.

The parameters from Tierney et. al. (2014) \cite{Tierney2014} that are relevant to our implementation of empty cargo flows which have not previously been introduced are as follows:

\begin{align*}
&V^{q^+}\subseteq V' &	&\textup{Set of visits with an equipment surplus of type $q$}\\
&V^{q^-}\subseteq V' &	&\textup{Set of visits with an equipment deficit of type $q$.}\\
&V^{q^*}\subseteq V' &	&\textup{Set of visits with an equipment surplus or deficit of type $q$ }\\
& & & \textup{$(V^{q^*} = V^{q^+} \cup V^{q^-})$.}\\
&r_q^{\textup{Var}} \in \mathbb{R}^+ &	&\textup{Revenue for each TEU of equipment of type $q \in Q$ delivered.}\\
\end{align*}

We denote the empty cargo flows with variables of the form $x_{sq}^{(o,d)}$, which is the amount of empty containers flowing from origin $o$ to destination $d$ of type $q$ on ship $s$. The lazy constraints are modified appropriately, so now we keep track of which demands and how many empty containers are on board, and if the capacity is violated a constraint is added to limit these flows.

\begin{equation}
\sum_{\substack{(o,d,q)\in M \\ Reach(i,j,o,d,q) \\ j \in Out(i)}} x_s^{(o,d,q)} + \sum_{q \in Q}\sum_{\substack{(o,d)\in V^{q+}\times V^{q-} \\ d\in \text{ReachFrom(o)}}} x_{sq}^{(o,d)} \leq u_s^{dc} \label{LCC3}
\end{equation}

Before presenting any results, we should note there is a problem with the public data set: it is not profitable to move empty containers. For all instances that have empty containers, the combined cost of picking up and dropping off an empty container is greater than the revenue earned from moving it. In three sets, the revenue value is 0, so even if there were no moving costs, there is no incentive to transport empty containers. Of all sets, the most profitable empty container opportunity will add -22 per container to the objective value.

As this leads to the trivial solution of $x^{(o,d)}_{sq} = 0$ for all empty cargo variables, we have decided to change the amount of revenue associated with moving empty cargo containers. All previous revenue values were either 0 or 150, but the moving costs are typically close to 150 each, so we have decided upon a new revenue value of 300. This ensures that at least moving empty containers does not result in a loss, and there are some profitable opportunities for moving empty containers, however not so profitable as to replace the movement of actual, profitable demands.

\begin{table}
	\centering
	\caption{Comparison of objective value for original instances and when revenue for empty cargo fixed at 300 units.}
	\label{t-EmptyCargoFlows}
	\begin{tabular*}{0.8\textwidth}{@{\extracolsep{\fill}}|c|c|c|c|c|}
		\hline & \multicolumn{2}{c|}{Without empty cargo} & \multicolumn{2}{c|}{With empty cargo} \\ \hline 
		Instance & Objective ($\times10^5$) & Time (s) & Objective ($\times10^5$) & Time (s) \\ \hline
		repos8p & -8.21 & 0.60 & -8.21 & 1.04\\ \hline 
		repos9p & -8.21 & 0.58 & -8.21 & 1.24\\ \hline 
		repos11p & 137.61 & 0.88 & 137.61 & 1.94\\ \hline 
		repos14p & 138.86 & 1.24 & 138.86 & 2.84\\ \hline 
		repos19p & 5.22 & 2.06 & 15.23 & 4.05\\ \hline 
		repos21p & -11.85 & 2.17 & -11.78 & 4.80\\ \hline 
		repos22p & -11.85 & 2.71 & 7.03 & 6.33\\ \hline 
		repos23p & 5.22 & 2.09 & 15.06 & 4.14\\ \hline 
		repos26p & -53.13 & 1.28 & -53.13 & 3.31\\ \hline 
		repos29p & -32.13 & 1.68 & -32.13 & 4.21\\ \hline 
		repos33p & -10.92 & 3.32 & -10.92 & 5.93\\ \hline 
		repos35p & 138.54 & 26.30 & 157.79 & 114.47\\ \hline 
		repos37p & 139.31 & 32.96 & 163.73 & 147.15\\ \hline 
		repos38p & 160.02 & 38.47 & 179.27 & 142.00\\ \hline 
		repos40p & 161.53 & 28.27 & 185.03 & 191.76\\ \hline 
	\end{tabular*}
\end{table}

\begin{table}
	\centering
	\caption{Number of lazy constraints added for instances with fixed empty cargo revenue of 300}\label{t-LazyAddedFixedRev}
	\begin{tabularx}{0.8\textwidth}{|c|Y|Y|Y|Y|}
		\hline  & \multicolumn{2}{|c|}{Number of Ships} & \multicolumn{2}{|c|}{Number of Constraints}  \\ 
		\hline Instance & Dry & Reef & Dry & Reef \\ 
		\hline repos19p & 5 & 0 & 15 & 0 \\ 
		repos22p & 6 & 0 & 21 & 0 \\ 
		repos23p & 5 & 0 & 12 & 0 \\
		repos35p & 9 & 1 & 101 & 1 \\
		repos37p & 9 & 1 & 173 & 1 \\
		repos38p & 9 & 1 & 94 & 1 \\
		repos40p & 9 & 1 & 188 & 4 \\ 
		\hline 
	\end{tabularx} 
\end{table}

The results of this can be seen in Table \ref{t-EmptyCargoFlows}. In some instances the increased revenue does not make a significant difference because profitable demands are already being moved, or there are no ship paths connecting $(o,d)$ pairs. However, for the instances with more empty cargo opportunities, carrying empty cargo leads to higher profits. Typically these empty containers only fill previously empty spaces rather than displacing demands. The biggest difference is in the number of lazy constraints generated. Table \ref{t-LazyAddedFixedRev} shows the instances which change when the revenue is increased. There are seven instances with empty cargo flows which do not add extra lazy constraints, however for some of the larger instances the number of constraints added approaches 200. This is still an insignificant number when compared to the size of the original formulation. All problem sets with empty cargo flows are solved using lazy constraints and column generation in less than four minutes. As such, it is no more difficult to consider the movement of empty containers in our formulation with column generation and lazy constraints.

\subsection{Flexible arcs}
There is one final feature of the Tierney et. al. (2014) \cite{Tierney2014} model that has not yet been implemented in our model: flexible arcs. Tierney et. al. note in their paper that the ``fuel consumption of a ship is approximately a cubic function of the speed of the vessel'', however in their model formulation the cost associated with sailing along a flexible arc is a linear function of sailing time, and thus has inverse relation with speed. Because the only constraints on the sailing time are upper bounds, the optimal solution for travelling along flexible arcs is always to sail as fast as possible, thus minimising the time spent sailing on the arc.

Another consideration is that flexible arcs only occur to connect flexible visits to the rest of the network. Flexible visits are ports which have no available demands, but do have a surplus or supply of empty cargo. This means that to use flexible arcs, the profit from moving empty cargo must be greater than the sum of the load and unload costs, plus the movement cost associated with the flexible arcs. For the original public data set, there are no profitable situations for moving empty cargo containers, including to flexible destinations.

\subsection{Comparison of model sizes}\label{SubSecModelSize}
The real benefit from lazy constraints comes from the dramatic reduction in model size, which allows the solver to manipulate the problem much more effectively. We can make a back-of-the-envelope estimate of the contribution of the demand flow variables to the size of the problem for the original formulation, and compare it to ours with column generation and lazy constraints.

In the original formulation, there are at most $|A'||M|$ variables of the form $x_{ij}^{(o,d,q)}$, where in our implementation there are less than $|M|$ variables in each sub-problem. Since there are only $|S|$ sub-problems, our formulation has at most $|S||M|$ variables. For the largest instance (repos44p), $|A'| \approx 13700$, $|M| = 1108$ and $|S| = 11$, so in our formulation we have tens of thousands of variables, where the original has tens of millions.

\begin{table}
	\centering
	\caption{Comparison of model sizes for instance repos44p}\label{t-ModelSize}
	\begin{tabular}{|c|c|c|c|}
		\hline & Rows & Columns & Non-zeros \\
		\hline Reduced MIP & 171102 & 269796 & 1858378 \\
		\hline Column Generation & \multirow{2}{*}{136391} & \multirow{2}{*}{76612} & \multirow{2}{*}{426025} \\
		sub-problems & & & \\
		\hline Col. Gen. and Lazy & \multirow{2}{*}{1627} & \multirow{2}{*}{6715} & \multirow{2}{*}{33475} \\
		Constraints sub-problems & & & \\
		\hline
	\end{tabular}
\end{table}

Looking next at the number of constraints governing the $x$ variables, the original formulation has at most $2|A'| + |M|\bullet(|V'|-1)$ constraints, where our formulation has fewer than $2|V'| + 3|M| + \gamma$, where $\gamma$ is the number of lazy constraints added. As seen in Table \ref{t-LazyAdded}, $\gamma < 40$ for even the largest instance. Again looking at repos44p, where $|V'| = 327$, our formulation has fewer than $4\times 10^3$ constraints, where the original has over $3.8\times 10^5$. This means our model is almost three orders of magnitude smaller than the original formulation, which is why Gurobi has no problem solving the largest instances.

Table \ref{t-ModelSize} is a comparison of the actual model sizes for instance repos44p between our reduced MIP, the sub-problems of the column generation implementation and of our reformulation using lazy constraints. For the latter two cases, the numbers are averaged over all 11 sub-problems. Remember that the models also include the $y_{ij}^s$ variables. These numbers are the reported model size before Gurobi's pre-solve stage, which typically removes up to half the rows from the column generation formulation, and a few hundred rows and columns from the lazy constraints formulation. This illustrates the significant difference that using lazy constraints has made, in that the number of rows, columns and non-zeros of the column generation sub-problems have been reduced by at least one order of magnitude each.

\section{Conclusion}
The LSFRP is a large and important problem in the shipping industry. The only feasible way to solve large instances of the problem is with column generation, however it is the application of lazy constraints to the model which makes the most difference. By dramatically reducing the size of the model, modern solver packages can manipulate the problem much more effectively, resulting in significantly reduced run times. 

This application of lazy constraints can be applied to any vehicle routing problem where transshipment cannot occur and the routes must be node-distinct. Further, it is likely that many more models can benefit from lazy constraints used in a similar way to how we have used them in the column generation sub-problems. Specifically, a large part of the model (in this case demand flow on each arc) can be optimistically approximated by a much smaller model (total amount of each demand), with additional refinements added using lazy constraints, as and when required. We consider this to be a rich area of further research.

\bibliographystyle{ormsv080} 
\bibliography{Transportation-Science} 

\end{document}